\documentclass{amsart}

\usepackage[normalem]{ulem}

\usepackage{lineno}

\usepackage{soul}

\usepackage{amssymb, amsmath, amsthm, hyperref, euscript, color}
\usepackage[dvipsnames]{xcolor}

\usepackage{amscd}

\usepackage{bbm}

\usepackage{enumitem}

\usepackage[english]{babel}

\numberwithin{equation}{section}

\theoremstyle{plain}

\newtheorem{theorem}{Theorem}

\newtheorem{corollary}{Corollary}

\newtheorem{proposition}{Proposition}

\newtheorem{remark}{Remark}

\newtheorem{thm}{Theorem}

\newtheorem{cor}[thm]{Corollary}

\theoremstyle{definition}

\newtheorem{definition}{Definition}

\newcommand{\R}{\mathbb{R}}

\newcommand{\Z}{\mathbb{Z}}

\newcommand{\N}{\mathbb{N}}

\newcommand{\C}{\mathbb{C}}

\newcommand{\1}{\quad |\quad}

\DeclareMathOperator{\Mod}{mod}
\DeclareMathOperator{\Dim}{dim}

\DeclareMathOperator{\Su}{SU}
\DeclareMathOperator{\OO}{O}
\DeclareMathOperator{\So}{SO}
\DeclareMathOperator{\Gl}{GL}

\begin{document}

\author{Marko Slapar and Rafael Torres}

\title[Totally real immersions/embeddings into $\C^N$]{Existence results of totally real immersions and embeddings into $\C^N$}

\address{Faculty of Education, University of Ljubljana\\Kardeljeva Po\u{s}\u{c}ad 16\\1000\\Ljubljana\\Slovenia}
\address{Faculty of Mathematics and Physics, University of Ljubljana\\Jadranska 19\\1000\\Ljubljana\\Slovenia}
\address{Institute of Mathematics, Physics and Mechanics\\Jadranska 19\\1000\\Ljubljana\\Slovenia}

\email{marko.slapar@pef.uni-lj.si}

\address{Scuola Internazionale Superiori di Studi Avanzati (SISSA)\\ Via Bonomea 265\\34136\\Trieste\\Italy}

\email{rtorres@sissa.it}


\maketitle

\emph{Abstract:} We prove that the existence of totally real immersions of manifolds is a closed property under cut-and-paste constructions along submanifolds including connected sums.  We study the existence of totally real embeddings for simply connected 5-manifolds and orientable 6-manifolds and determine the diffeomorphism and homotopy types. We show that the fundamental group is not an obstruction for the existence of a totally real embedding for high-dimensional manifolds in contrast with the situation in dimension four.

\section{Introduction and main results}

In this paper, we are interested in the following kind of maps.

\begin{definition}\label{Definition TotallyReal} Let $M^N$ be a closed smooth $N$-manifold and let $J$ denote the standard complex structure on the tangent bundle of $\C^N$. An immersion $M^N\rightarrow \C^N$ is totally real if if the tangent bundle $TM^N$ contains no complex lines, i.e., if\begin{equation}\label{Totally}TM^N\cap JTM^N = \{0\}\end{equation} at all points of $TM^N$. An embedding $M^N \hookrightarrow \C^N$ that satisfies (\ref{Totally}) is called a totally real embedding.
\end{definition}

A canonical problem is to distinguish between those manifolds that admit totally real immersions and the ones who admit totally real embeddings. This dichotomy is already somewhat interesting in the case of manifolds with simple topology. For example, every $N$-sphere $S^N$ admits a totally real immersion into $\C^N$, yet no totally real embedding exists if $N > 3$; see Gromov \cite[p. 193]{[Gromov2]}, Stout-Zame \cite{[StoutZame]}. We occupy ourselves with the study of the distinction of the maps of Definition \ref{Definition TotallyReal} for a large class of manifolds in this paper. Necessary and sufficient topological conditions for the existence of totally real immersions and embeddings have been studied by Gromov \cite{[Gromov1]}, Wells \cite{[Wells]}, Audin \cite{[Audin]}, Forstneri\u{c} \cite{[Forstneric]}, Gong \cite{[Gong]} and Jacobowitz-Landweber \cite{[JL]} among several other mathematicians (see Section \ref{Section TRI}).

Our first two theorems state that the existence of a totally real immersion is a property which is closed under certain fundamental cut-and-paste constructions of manifolds along submanifolds using a trivial framing; please see Remark \ref{Remark Framing}. The reader is directed to \cite[Section 2]{[MilnorKervaire]} and \cite[Section 1]{[Milnor]} for the precise definitions of the cut-and-paste operations that we use in this paper. 

\begin{thm}\label{Theorem Surgeries}Let $M^N$ be a closed smooth $N$-manifold that admits a totally real immersion into $\C^N$ and let $\imath: S^p\times D^{N - p}\hookrightarrow M^N$ be a smooth embedding for $0\leq p \leq 2$. The $N$-manifold\begin{equation}\hat{M}^N:= M^N \backslash \imath(S^p\times D^{N - p})\cup (D^{p + 1}\times S^{N - p - 1})\end{equation}that is obtained by performing surgery along $\imath(S^p\times \{0\})$ admits a totally real immersion $\hat{M}^N\rightarrow \C^N$.
\end{thm}

\begin{thm}\label{Theorem ConnectedSums}Let $M_1^N$ and $M_2^N$ be closed smooth oriented $N$-manifolds that admit a totally real immersion into $\C^N$. There is a totally real immersion $M_1^N\#\overline{M_2^N}\rightarrow \C^N$.

\end{thm}

A circle has a totally real embedding into $\C$ and every closed orientable surface admits a totally real immersion into $\C^2$, yet the 2-torus is the only orientable closed surface that admits a totally real embedding. Ahern-Rudin \cite{[AhernRudin]} constructed an explicit totally real embedding $S^3\hookrightarrow \C^3$ and Forstneri\u{c} showed that every closed orientable 3-manifold admits a totally real embedding into $\C^3$ \cite[1.4 Theorem]{[Forstneric]}. Jacobowitz-Landweber have shown that a necessary and sufficient condition for a closed smooth orientable 4-manifold $M^4$ to admit a totally real immersion into $\C^4$ is the vanishing of its first Pontrjagin class $p_1(M^4)$ \cite[Corollary 4.1]{[JL2]}. Our next result address the situation in dimension five using the classification of closed simply connected 5-manifolds of Barden \cite{[Barden]} and Smale \cite{[Smale]}. A smooth manifold $M^N$ is irreducible if for every connected sum decomposition $M^N = M_1\# M_2$, either $M_1$ or $M_2$ is diffeomorphic to the $n$-sphere. The nontrivial 3-sphere bundle over the 2-sphere is denoted by $S^3\widetilde{\times} S^2$.

\begin{thm}\label{Theorem 5D} Every closed smooth simply connected 5-manifold $M^5$ admits a totally real immersion\begin{equation}M^5\rightarrow \C^5.\end{equation} Let $M^5$ be an irreducible simply connected 5-manifold. There is a totally real embedding\begin{equation}M^5\hookrightarrow \C^5\end{equation}if and only if\begin{equation}M^5\in \{\Su(3)/\So(3), S^2\times S^3, S^3\widetilde{\times} S^2\}\end{equation} up to diffeomorphism.
\end{thm}

The 5-sphere is the only irreducible simply connected 5-manifold that does not admit a totally real embedding into complex 5-space.  A complete list of simply connected 5-manifolds that admit such an embedding is given in Section \ref{Section Proof5D} and it includes the following set of examples. 

\begin{cor}\label{Corollary 5D} Let $M^5$ be a closed smooth simply connected 5-manifold with torsion-free second homology group $H_2(M^5; \Z)$ and suppose $k\in \N$. There is a totally real embedding\begin{equation}M^5\hookrightarrow \C^5\end{equation}if and only if $M^5$ is diffeomorphic to\begin{equation}S^3\widetilde{\times} S^2\# (2k - 2)(S^2\times S^3)\end{equation} if $w_2(M^5) \neq 0$ and to\begin{equation}(2k - 1)(S^2\times S^3)\end{equation} otherwise.
\end{cor}

The homotopy type of a closed simply connected 5-manifold determines its diffeomorphism class \cite[Section 2]{[Barden]}. This is no longer the case in dimension six, where closed homotopy equivalent 6-manifolds need not be homeomorphic. Building on results of Wall \cite{[Wall]}, our next theorem states in terms of characteristic classes, the necessary and sufficient conditions for the existence of the maps of Definition \ref{Definition TotallyReal} for 6-manifolds. The symbol $\chi(M^N)$ stands for the Euler characteristic of the $N$-manifold $M^N$.

\begin{thm}\label{Theorem 6D} Let $M^6$ be a closed smooth orientable 6-manifold without 2-torsion in $H^3(M; \Z)$. There is a totally real immersion\begin{equation}\label{6D TRImmersion}M^6\rightarrow \C^6\end{equation} if and only if the first Pontrjagin class satisfies $p_1(M^6) = 0$.

There is a totally real embedding\begin{equation}M^6\hookrightarrow \C^6\end{equation} if and only if\begin{equation}p_1(M^6) = 0 = \chi(M^6).\end{equation}
\end{thm}

Results of Dehn \cite{[Dehn]}, Kervaire-Milnor \cite{[MilnorKervaire]} and Gromov \cite{[Gromov1]} imply that there is a totally real immersion $M^N(G)\rightarrow \C^N$ of a closed orientable $N$-manifold $M^N(G)$ with prescribed finitely presented fundamental group $\pi_1(M^N(G)) = G$ for every $N\geq 4$. However, the fundamental group does impose a restriction for the existence of a totally real embedding of a 4-manifold into $\C^4$. Indeed, an argument due to Wells \cite{[Wells]} shows that the Euler characteristic of such a 4-manifold must be zero, while Poincar\'e duality shows that there is a myriad of choices of finitely presented groups $G$ that force the Euler characteristic of a closed orientable 4-manifold $M^4(G)$ to be strictly positive. Our next result shows that this is not the case in higher dimensions.

\begin{thm}\label{Theorem G} Let $G$ be a finitely presented group. There is a totally real immersion\begin{equation}M^4(G)\rightarrow \C^4\end{equation} of a given closed smooth orientable 4-manifold $M^4(G)$ with fundamental group isomorphic to $G$. 

There is a totally real embedding\begin{equation}M^N(G)\hookrightarrow \C^N,\end{equation}where $M^N(G)$ is a given closed smooth $N$-manifold with fundamental group isomorphic to $G$ and for every $N\geq 5$.
\end{thm}

The structure of the paper is as follows. Classical and new existence results on totally real immersions and embeddings which we build upon to prove the results presented in the introduction are collected in Section \ref{Section TRI}. A proof of Theorem \ref{Theorem Surgeries} and Theorem \ref{Theorem ConnectedSums} is given in Section \ref{Section SurgeryConstructions}. Theorem \ref{Theorem 5D} and its corollary are proven in Section \ref{Section Proof5D} and Section \ref{Section Corollary5D}, respectively, while Theorem \ref{Theorem 6D} is proven in Section \ref{Section Proof6D} and Theorem \ref{Theorem G} in Section \ref{Section ProofTheoremG}. Section \ref{Section Generic} contains some results addressing the existence of more general immersions with respect to the ones considered in Definition \ref{Definition TotallyReal} (cf. \cite{[JL]}).

\section{Totally real immersions and embeddings}\label{Section TRI}

Wells showed that if an $N$-manifold $M^N$ admits a totally real immersion into $\C^N$ in the sense of Definition \ref{Definition TotallyReal}, then its complexified tangent bundle\begin{equation}\C TM^N:= TM^N\otimes \C\end{equation} is trivial \cite{[Wells]}. Gromov proved that if $\C TM$ is trivial, then such an immersion exists by using convex integration in \cite{[Gromov1]} (cf. \cite[Theorem 1.2]{[JL]}).

\begin{theorem}\label{Theorem TRImmersions} Wells \cite{[Wells]}, Gromov \cite{[Gromov1]}. There is a totally real immersion $M^N\rightarrow \C^N$ if and only if $\C TM^N$ is  a trivial bundle  of rank $N$.
\end{theorem}

Let $\epsilon^k$ be a trivial rank $k$ bundle over $M^N$. A smooth $N$-manifold $M^N$ is stably-parallelizable if the Whitney sum $TM^N\oplus \epsilon^1$ is a trivial bundle. If the bundle $TM^N\oplus \epsilon^1$ is trivial, then so is its complexification $\C TM^N$ \cite[Lemma 1.2]{[JL]}. Therefore, Theorem \ref{Theorem TRImmersions} has the following immediate consequence. 

\begin{corollary}\label{Corollary Stably} Every stably parallelizable manifold $M^N$ admits a totally real immersion into $\C^N$.
\end{corollary}

Audin studied conditions under which the product of manifolds admits a totally real immersion whenever each of the factors does \cite[6.2]{[Audin]}. 

\begin{proposition}\label{Proposition Audin} Audin \cite[6.2.3. Remarque]{[Audin]}. Suppose there exist a totally real immersion $X^N\rightarrow \C^N$ and a totally real embedding $Y^M\rightarrow \C^M$. There is a totally real embedding\begin{equation}X^N\times Y^M\hookrightarrow \C^N\times \C^M\cong \C^{N + M}.\end{equation}
\end{proposition}

Our next result extends her results to a more general case of fiber bundles.

\begin{proposition}\label{Proposition GBundle} Let $X^N$ be a closed smooth orientable $N$-manifold with trivial complexified tangent bundle $\C TX^N$. Let $Y^{N + k}$ be the total space of a principal $k$-torus bundle for $k\in \N$\begin{equation}T^k\hookrightarrow Y^{N + k}\overset{\pi}\longrightarrow X^N.\end{equation} There exists a totally real immersion\begin{equation}Y^{N + k}\rightarrow \C^{N + k}.\end{equation}
\end{proposition}

Notice that the converse to the conclusion of Proposition \ref{Proposition GBundle} does not hold as exemplified by $S^1\hookrightarrow S^5\rightarrow \mathbb{CP}^2$. Moreover, the conclusion of Proposition \ref{Proposition GBundle} can be strengthened to cover embeddings.

\begin{proof} The tangent bundle of $Y^{N + k}$ can be written as\begin{equation}TY^{N + k} \cong \pi^{\ast}TX^N\oplus\epsilon^k\end{equation} where $\epsilon^k$ is a trivial rank $k$ bundle over $Y^{N + k}$ whose elements are tangent to the torus fibers \cite{[Szczarba]}. Its complexification is\begin{equation}\C TY^{N + k} = TY^{N + k}\otimes \C = (\pi^{\ast}TX^N\oplus \epsilon^k)\otimes \C = (\pi^{\ast}TX^N\otimes \C) \oplus (Y^{N + k}\times \C),\end{equation} which is trivial since $\C TX^N$ is assumed to be trivial. Theorem \ref{Theorem TRImmersions} implies that there is a totally real immersion $Y^{N + k}\rightarrow \C^{N + k}$.
\end{proof}

Recall that the Kervaire semi-characteristic of a closed smooth $n$-manifold $M^N$ of dimension $N = 2k + 1$ for $k\in \N$ is defined as\begin{equation}\label{SemiChar}\hat{\chi}_{\Z/2}(M^N):= \sum^k_{i = 0}\Dim H^i(M^N; \Z/2) \Mod 2\end{equation}by Luzstig-Milnor-Peterson \cite{[LMP]}. 

\begin{theorem}\label{Theorem AudinEven} Audin \cite[0.4 Proposition, 0.5 Th\'eor\`eme, 0.6 Corollaire]{[Audin]}. Let $M^N$ be a closed smooth connected orientable $N$-manifold of dimension and  suppose there is a totally real immersion $M^N\rightarrow \C^N$. 

(A) Suppose $N$ is even. There is a totally real embedding $M^N\hookrightarrow \C^N$ if and only if $\chi(M^N) = 0$.

(B) Suppose $N = 4k + 1$ for $k\in \N$. There is a totally real embedding $M^N\hookrightarrow \C^N$ if and only if $\hat{\chi}_{\Z/2}(M^N) = 0$.
\end{theorem}

We finish the section with the following sets of examples.

\begin{proposition}\label{Proposition 5D} Let $M^5$ be a closed smooth simply connected 5-manifold with second Stiefel-Whitney class $w_2(M^5) = 0$. 

The manifold $M^5$ is stably parallelizable and there is a totally real immersion $M^5\rightarrow \C^5$.

There is a totally real embedding $M^5\hookrightarrow \C^5$ if and only if $M^5$ is parallelizable. 
\end{proposition}

\begin{proof} A result of Hirsch states that an $n$-manifold $M$ is stably parallelizable if and only if $M$ is orientable and it immerses into $\R^{n + 1}$ \cite{[Hirsch1]} (cf. \cite[Section 1.1.3]{[Gromov2]}). Barden has shown a closed simply connected 5-manifold admits an immersion into $\R^6$ if and only if its second Stiefel-Whitney class vanishes \cite[Lemma 2.4]{[Barden]}. The existence of the totally real immersion now follows from Corollary \ref{Corollary Stably}.

Item (B) of Theorem \ref{Theorem AudinEven} implies that there exists a totally real embedding of $M^5$ into $\C^5$ if and only if the Kervaire semi-characteristic $\hat{\chi}_{\Z/2}(M^5) = 0$. Kervaire has shown in \cite{[Kervaire1]} that the only obstruction for a stably parallelizable odd-dimensional manifold to be parallelizable, is the vanishing of the Kervaire semi-characteristic provided that the dimension is not one, three, nor seven.

\end{proof}

Smale's \cite{[Smale]} and Barden's \cite{[Barden]} classification of closed simply connected 5-manifolds implies that a manifold of Proposition \ref{Proposition 5D} is diffeomorphic to a connected sum\begin{equation}\label{Prototype}S^5\#(k - 1)(S^2\times S^3)\# (k_1 - 1)M_{p_1^k}\#\cdots \#(k_i - 1)M_{p_i^k}\end{equation}for $k, k_1, \ldots, k_j\in \N$, where each manifold $M_{p_i^k}$ has $H^2(M_{p_i^k}; \Z) = \Z/p_i^k\oplus \Z/p_i^k$ and $w_2(M_{p_i^k}) = 0$; see Table \ref{Table 5-manifolds}.  In particular, the manifold (\ref{Prototype}) admits a totally real embedding into complex 5-space if and only if $k$ is a positive even number.

Appealing to another classification result due to Smale \cite{[Smale]} we obtain the following six-dimensional examples.

\begin{proposition}\label{Example 6D} Every closed smooth 2-connected 6-manifold $M^6$ admits a totally real immersion\begin{equation}\label{ExampleImmersion6D}M^6\rightarrow \C^6\end{equation}and there is a totally real embedding\begin{equation}\label{ExampleEmbedding6D}M^6\hookrightarrow \C^6\end{equation} if and only if $M^6$ is diffeomorphic to $S^3\times S^3$. 

\end{proposition}

\begin{proof} Smale \cite{[Smale]} has shown that a closed smooth 2-connected 6-manifold $M$ is diffeomorphic to $S^6$ or to a connected sum $n(S^3\times S^3)$ of $n$ copies of the product of two 3-spheres for $n\in \N$. Since the connected sum of two stably parallelizable manifolds is stably parallelizable, it follows that every closed smooth 2-connected 6-manifold is stably parallelizable. Corollary \ref{Corollary Stably} implies that the totally real immersion (\ref{ExampleImmersion6D}) exists for every such 6-manifold. Theorem \ref{Theorem AudinEven} and Smale's cited classification result imply that there is a totally real embedding (\ref{ExampleEmbedding6D}) if and only if $M^6$ is diffeomorphic to $S^3\times S^3$. We point out that Ahern-Rudin have given an explicit construction of a totally real embedding $S^3\hookrightarrow \C^3$ \cite{[AhernRudin]} that can be used to construct a totally real embedding of the product of two 3-spheres into $\C^6$ (cf. \cite[6.2.3 Remarque]{[Audin]}). 

\end{proof}

\section{Proofs}

\subsection{Proof of Theorems \ref{Theorem Surgeries} and Theorem \ref{Theorem ConnectedSums}}\label{Section SurgeryConstructions}We first show that the manifold\begin{equation}\hat{M}^N:= M^N \backslash \imath(S^2\times D^{N - 2})\cup (D^{3}\times S^{N - 3})\end{equation}that is obtained by performing surgery along $\imath(S^2\times \{0\})$ admits a totally real immersion $\hat{M}^N\rightarrow \C^N$ for clarity purposes, and then discuss the corresponding generalization to the cases $p = 0, 1$. Theorem \ref{Theorem TRImmersions} states that we need to show that $\C T{\hat{M}}^N$ is trivial. Set $S:= \imath(S^2\times\{0\})$ and let $\{e_1,e_2,e_3,\ldots,e_N\}$ be sections that trivialize the bundle $\C TM^N$. We can assume that along the 2-sphere $S$, the elements $\{e_1,e_2\}$ give a trivialization of $\C TS$ and the elements $\{e_3,\ldots,e_N\}$ trivialize the normal bundle $NS$ of $S$. This is justified by the following argument. Let $\{f_1,f_2,f_3,\ldots,f_N\}$ be nowhere zero sections of $\C TM^N|_S$, defined over $S$, so that $\{f_1,f_2\}$ trivialize $\C TS$ and $\{f_3,\ldots,f_N\}$ trivialize $NS$. There exists a map\begin{equation}A:S\to \Gl(N,\C),\end{equation} so that $f_i = A e_i$ for $i=1,2,\ldots,n$. Since $\pi_2(\Gl(N,\C))=0$, there is a homotopy $A_t$, so that $A_t=A$ near $t=0$ and $A_t=I$ near $t=1$. We can use this homotopy to connect the trivialization $\{f_1,f_2,f_3,\ldots,f_N\}$ over $S$  with the trivialization $\{e_1,e_2,e_3,\ldots,e_N\}$ outside a neighborhood of $S$. 

Let now $W$ be the standard cobordism between $M^N$ and $\hat{M}^N$ that is obtained by attaching an $N + 1$-dimensional $3$-handle $H:= D^3\times D^{N - 2}$ to $M^N\times [0,1]$ along $M^N\times \{1\}$ via the gluing map $i_{S^2\times D^{N - 2}}\times \{1\}.$ We proceed to show that $\C TW$ is a trivial bundle. Let $e$ be inward normal vector field to the boundary sphere $S = S^2\times\{0\}$ in the core $D:= D^3\times \{0\}$ of the handle. Since $\pi_2(\Gl(3,\C))=0$, we can extend $\{e,e_1,e_3\}$ to a trivialization of $\C TD$. We can also trivially extend $\{e_3,\ldots,e_N\}$ to a nonzero trivialization of the normal bundle of $D$ in the handle $H$; recall $\pi_2(\OO (N - 2)) = 0$ and there are no framing issues attaching $3$-handles. Since $M\cup D$ is a deformation retract of $W$,  $\{e, e_1,e_2,e_3,\ldots,e_N\}$ can be extended to give a trivialization of $\C TW$. Since along $\hat{M}^N = \partial_+W$, the tangent bundle $TW$ is a Whitney sum $T\hat{M}^N\oplus \epsilon^1$ with a trivial line bundle $\epsilon^1$, we conclude that $\C T\hat{M}^N\oplus (\epsilon^1 \otimes \C)$ is trivial. Since $\C T\hat{M}^N$ is stably trivial, then it is a trivial bundle \cite [Lemma 1.2]{[JL]}.

\begin{remark}\label{Remark Framing} Small tweaks to the proof of Theorem \ref{Theorem Surgeries} yield the same conclusion for surgeries performed to $M^N$ along \begin{equation}S^p\times D^{n-p}\hookrightarrow M^n\end{equation} for $p = 0, 1$ whose normal bundle over the handle is trivial. In these cases there are two choices of framings $\pi_0(\OO (N))=\pi_1(\OO (N - 1))=\Z/2$ (except for $N = 3$ when $\pi_1(\OO (2))=\Z$). For exactly one choice of framing, we can extend the trivialisation of the normal bundle from the boundary of the core of the handle to its core as in the previous proof. 
\end{remark} 

The above remark in the case $p=0$ yields Theorem \ref{Theorem ConnectedSums} as a corollary.

\subsection{Proof of Theorem \ref{Theorem 5D}}\label{Section Proof5D} The case of 5-manifolds with vanishing second Stiefel-Whitney class was settled in Proposition \ref{Proposition 5D}. The classification results of closed simply connected 5-manifolds up to diffeomorphism of Barden \cite{[Barden]} and Smale \cite{[Smale]} imply that any such 5-manifold is difeomorphic to a connected sum of manifolds in Table \ref{Table 5-manifolds}. We proceed to argue that every manifold in the table has trivial complexified tangent bundle and then invoke Theorem \ref{Theorem ConnectedSums} and Theorem \ref{Theorem TRImmersions} in order to prove the first part of Theorem \ref{Theorem 5D}.  Therefore, we need to show that the Wu manifold $\Su(3)/\So(3)$, the nontrivial bundle $S^3\widetilde{\times} S^2$, and the manifold $X_k$ with $k$ have trivial complexified tangent bundle. Audin has shown that the Wu manifold admits a totally real embedding into $\C^5$ \cite[Proposition 0.8]{[Audin]}. The same conclusion holds for $S^3\widetilde{\times} S^2$ by Proposition \ref{Proposition GBundle} since it is the total space of a circle bundle\begin{equation}S^1\hookrightarrow S^3\widetilde{\times} S^2\rightarrow \mathbb{CP}^2\#\overline{\mathbb{CP}^2},\end{equation} where the base 4-manifold has trivial complexified tangent bundle by \cite[Corollary 4.1]{[JL2]} given that its first Pontrjagin class is zero. We now show that $\C TX_k$ is a trivial bundle for every value $k\in \N$. Let $S\hookrightarrow S^3\widetilde{\times} S^2$ be a 2-sphere that represents $2^k$-times the generator of the infinite cyclic group $H_2(S^3\widetilde{\times} S^2; \Z)$. Notice that the normal bundle of $S$ is trivial. The manifold $X_k$ is obtained from $S^3\widetilde{\times} S^2$ by performing the surgery of Theorem \ref{Theorem Surgeries} along $S$ and therefore $\C TX_k$ is trivial. The claim regarding totally real embeddings into complex 5-space follows immediately from Theorem \ref{Theorem AudinEven}.

\hfill $\square$

\begin{remark} Every closed simply connected 5-manifold is diffeomorphic to a connected sum of manifolds that are listed in Table \ref{Table 5-manifolds}. Such a 5-manifold admits a totally real embedding into complex 5-space if and only if it is diffeomorphic to\begin{equation}S^5\# \delta (S^3\widetilde{\times}S^2)\#(k_1 - 1)(S^2\times S^3)\#(k_2 - 1)(\Su(3)/\So(3))\#M\#N\end{equation}

for\begin{equation}\delta + k_1 + k_2 + - 1 = 0\end{equation}where $\delta\in \{0, 1\}$, $k_1, k_2\in \N$, the manifold $M$ is a connected sum of an arbitrary number of copies of $M_p^k$ and $N$ is a connected sum of an arbitrary number of copies of the manifold $X_k$ of Table \ref{Table 5-manifolds}. If the second Stiefel-Whitney class is zero, the explicit diffeomorphism type is given in (\ref{Prototype}).
\end{remark}

\begin{table}
\begin{tabular}{|l|c|r|r|} \hline 
5-manifold $Y$ & $H_2(Y; \Z)$ & $w_2(Y)$ & $\hat{\chi}_{\Z/2}(Y) $\\ \hline \hline
   $S^5$ & 0 & 0  & 1 \\ \hline
  $S^2\times S^3$ & $\Z$ & 0  & 0\\ \hline
   $M_{p^k}$ & $\Z/p^k\oplus \Z/p^k$ & 0  & 1\\ \hline
     $\Su(3)/\So(3)$ & $\Z/2$ & $\neq 0$  & 0\\ \hline
    $S^3\widetilde{\times} S^2$ & $\Z$ & $\neq 0$ & 0\\ \hline
 $X_{k}$ & $\Z/2^k\oplus \Z/2^k$ & $\neq 0$ & 1 \\ \hline

\end{tabular}

\caption{Building blocks of simply connected 5-manifolds ($k\in \N$)}
\label{Table 5-manifolds}
\end{table}

\subsection{Proof of Corollary \ref{Corollary 5D}}\label{Section Corollary5D} A 5-manifold that satisfies the hypothesis of the corollary is diffeomorphic to $S^5$ or to a connected sum of copies of $S^2\times S^3$ if its second Stiefel-Whitney class is zero and to a connected sum of $S^3\widetilde{\times} S^2$ with copies of $S^2\times S^3$ if its second Stiefel-Whitney class is not zero according to Barden and Smale aforementioned classification results. Using a Mayer-Vietoris sequence and the universal coefficients theorem \cite{[Hatcher]}, it is immediate to compute\begin{equation}\hat{\chi}_{\Z/2}(S^3\widetilde{\times} S^2\# (n - 1)(S^2\times S^3)) = n + 1 \Mod 2\end{equation}and\begin{equation}\hat{\chi}_{\Z/2}(n(S^2\times S^3)) = n + 1 \Mod 2\end{equation}for $n\in \N$. Item (B) of Theorem \ref{Theorem AudinEven} implies that a sufficient and necessary condition for the totally real embedding to exists is for $n$ to be an odd natural number.

\subsection{Proof of Theorem \ref{Theorem 6D}}\label{Section Proof6D}Set $M:= M^6$ and suppose there exists such a totally real immersion into $\C^6$. The triviality of the complexified tangent bundle $\C TM$ implies that $c_2(\C TM) = 0$. By definition of the Pontrjagin classes\begin{equation}\label{Definition Pontrjagjn}p_i(M) = c_{2i}(\C TM),\end{equation} it follows that $p_1(M) = 0$ if and only if $c_2(\C TM) = 0$. A closed orientable 6-manifold $M$ without 2-torsion in $H^3(M; \Z)$ admits an almost-complex structure \cite[Proposition 8]{[OkonekVV]}, (cf. \cite[Section 7]{[Wall]}). The complexified tangent bundle of an almost-complex manifold $M$ has the canonical eigenspaces decomposition\begin{equation}\label{Isomorphism2}\C TM = T^{1, 0}M\oplus T^{0, 1}M,\end{equation}into the holomorphic $T^{1, 0}M$ and antiholomorphic $T^{0, 1}M$ tangent bundles of $M$. The Chern classes of these bundles satisfy the equality\begin{equation}\label{EqualityC1}c_i(T^{1, 0}M) = (-1)^i c_i(T^{0, 1}M).\end{equation}  Suppose now that $p_1(M) = 0$. To prove the converse, we claim\begin{equation}\label{Claim1}c_i(\C TM) = 0\end{equation} for $i\in \{1, 2, 3\}$. The conclusion of Theorem \ref{Theorem 6D} readily follows from (\ref{Claim1}) since a complex vector bundle of rank greater or equal to 3 over a closed orientable 6-manifold is trivial if and only if its Chern classes $\{c_1, c_2, c_3\}$ vanish \cite[p. 416]{[Peterson]}. We proceed to show that these characteristic classes are zero. From (\ref{Definition Pontrjagjn}), we see that our hypothesis implies $c_2(\C TM) = 0$. Let us now show that $c_1(\C TM)$ is zero. It follows from using the Whitney product formula for Chern classes and (\ref{Isomorphism2}) that\begin{equation}c_1(\C TM) = c_1(T^{1, 0}M\oplus T^{0, 1}M) = c_1(T^{1, 0}M) + c_1(T^{0, 1}M).\end{equation}Identity (\ref{EqualityC1}) implies $c_1(\C TM) = 0$. Similarly, the Whitney product formula and identity (\ref{EqualityC1}) imply\begin{equation}c_3(\C TM) = c_1(T^{1, 0}M)\cup c_2(T^{0, 1}M) + c_2(T^{1, 0}M)\cup c_1(T^{0, 1}M) = 0.\end{equation}We conclude that (\ref{Claim1}) holds and it follows that $\C TM$ is trivial. Theorem \ref{Theorem TRImmersions} implies the existence of the totally real immersion (\ref{6D TRImmersion}) as claimed.

The claims about the existence of a totally real embedding now follow from the first part of the theorem and Item (A) of Theorem \ref{Theorem AudinEven}.

\hfill $\square$


\subsection{Proof of Theorem \ref{Theorem G}: examples with arbitrary fundamental group}\label{Section ProofTheoremG} Classical results of Dehn \cite{[Dehn]} and Kevaire-Milnor \cite{[MilnorKervaire]} imply that for any $N\geq 4$ and any finitely presented group $G$ there exists a closed smooth stably parallelizable $N$-manifold $M^N(G)$ such that the fundamental group $\pi_1(M^N(G))$ is isomorphic to $G$. In particular, the complexified tangent bundle $\C T M^N(G)$ is trivial  and a result of Gromov \cite{[Gromov1]} implies that there is a totally real immersion $M^N(G)\rightarrow \C^N$ for every $N\geq 4$ as it is stated in Theorem \ref{Theorem TRImmersions}. We now proceed to show the existence of a totally real embedding into $\C^N$. For the values $N = 4 + 2(k_1 - 1) + 3k_2$ with $k_1, k_2\in \N$ the claim follows immediately by invoking Proposition \ref{Proposition Audin} on the product of the 4-manifold $M^4(G)$ with $(k_1 - 1)$ copies of $S^2$ and $k_2$ copies of $S^3$. We apply Theorem \ref{Theorem AudinEven} to deal with the cases $N$ even and $N = 4 k + 1$. A standard computation using a Mayer-Vietoris sequence and the universal coefficients theorem (see \cite{[Hatcher]}) implies that either $\hat{\chi}_{\Z/2}(M^5(G)) = 0$ or $\hat{\chi}_{\Z/2}(M^5(G)\# S^2\times S^3) = 0$. Since $\pi_1(M^5(G)) = \pi_1(M^5(G)\# S^2\times S^3)$, we conclude that the claim for $N = 5$ holds. Using product manifolds as before, we conclude that the claim regarding the existence of a totally real embedding holds for odd $N$. For even values of $N$, the argument is similar. A Mayer-Vietoris sequence reveals that the Euler characteristic of $M^N(G)$ is an even number. Taking connected sums of $M^N(G)$ with copies of $S^3\times S^{N - 3}$ and $S^2\times S^{N - 2}$, one obtains a manifold with fundamental group $G$ and zero Euler characteristic, which we continue to call $M^N(G)$.

\hfill $\square$

\begin{remark} Johnson-Walton \cite[Theorem A]{[JW]} pointed out that work of Kervaire  \cite{[Kervaire1], [Kervaire2]} implies that the manifolds of Theorem \ref{Theorem G} are parallelizable.\end{remark}

\section{Examples of generic immersions}\label{Section Generic}

In this last section, we mention some examples of the following kind of immersions.

\begin{definition}\label{Definition Generic}Let $n$ be a nonnegative integer and $k\in \N$. An immersion\begin{equation}\label{EquationGeneric}\pi: M^{2n + k}\rightarrow \C^{n + k}\end{equation}is said to be generic if at each point $p\in M$ the real vector space\begin{equation}\pi_{\ast}TM\cap J\pi_{\ast}M\end{equation}has dimension $2n$.
\end{definition} 

Definition \ref{Definition Generic} recovers the notion of totally real immersions of Definition \ref{Definition TotallyReal} for the values $(n, k) = (0, N)$. A result of Hirsch \cite{[Hirsch1]} states that the bundle $TM^{2n + 1}\oplus \epsilon$ is trivial if and only if $M^{2n + 1}$ immerses in $\R^{2n + 2}$. Since every real hypersurface in $\C^{n + 1}$ is automatically generic, any closed stably parallelizable $(2n + 1)$-manifold has a generic immersion into  $\C^{n + 1}.$ (cf. \cite[Remark 4]{[JL]}). This yields the following examples of generic immersions.

\begin{corollary}\label{Corollary GenericImmer} A closed simply connected 5-manifold has a generic immersion into $\C^3$ if and only if its second Stiefel-Whitney class is zero, i.e., if the 5-manifold is stably parallelizable. 

For every $n\geq 2$ and every finitely presented group $G$ there is a generic immersion\begin{equation}M^{2n + 1}(G)\rightarrow \C^{n + 1}\end{equation} of a closed smooth orientable $(2n + 1)$-manifold $M^{2n + 1}(G)$ with fundamental group isomorphic to $G$. 
\end{corollary}

\end{document}